\documentclass[12pt]{amsart}
\usepackage{amsfonts, amssymb, amsthm, amsmath, array, braket, hyperref, mathtools, mathrsfs, enumitem}
\usepackage[table]{xcolor}
\usepackage[all]{xy}
\usepackage{aliascnt}

\usepackage[nameinlink]{cleveref}
\newtheorem{theorem}{Theorem}[section]

\newaliascnt{lemma}{theorem}
\newtheorem{lemma}[lemma]{Lemma}
\aliascntresetthe{lemma}

\newaliascnt{corollary}{theorem}
\newtheorem{corollary}[corollary]{Corollary}
\aliascntresetthe{corollary}

\newaliascnt{proposition}{theorem}
\newtheorem{proposition}[proposition]{Proposition}
\aliascntresetthe{proposition}

\newaliascnt{thm}{theorem}

\aliascntresetthe{thm}

\newaliascnt{lem}{lemma}

\aliascntresetthe{lem}

\newaliascnt{cor}{corollary}

\aliascntresetthe{cor}

\newaliascnt{prop}{proposition}

\aliascntresetthe{prop}

\theoremstyle{definition}
\newaliascnt{definition}{theorem}
\newtheorem{definition}[definition]{Definition}
\aliascntresetthe{definition}

\newaliascnt{example}{theorem}
\newtheorem{example}[example]{Example}
\aliascntresetthe{example}

\theoremstyle{remark}
\newaliascnt{remark}{theorem}
\newtheorem{remark}[remark]{Remark}
\aliascntresetthe{remark}

\crefname{theorem}{theorem}{theorems}
\Crefname{theorem}{Theorem}{Theorems}

\crefname{lemma}{lemma}{lemmas}
\Crefname{lemma}{Lemma}{Lemmas}
\crefname{lem}{lemma}{lemmas}
\Crefname{lem}{Lemma}{Lemmas}

\crefname{corollary}{corollary}{corollaries}
\Crefname{corollary}{Corollary}{Corollaries}
\crefname{cor}{corollary}{corollaries}
\Crefname{cor}{Corollary}{Corollaries}

\crefname{proposition}{proposition}{propositions}
\Crefname{proposition}{Proposition}{Propositions}
\crefname{prop}{proposition}{propositions}
\Crefname{prop}{Proposition}{Propositions}

\crefname{definition}{definition}{definitions}
\Crefname{definition}{Definition}{Definitions}

\crefname{remark}{remark}{remarks}
\Crefname{remark}{Remark}{Remarks}

\newcommand{\defi}[1]{\textsf{#1}} 

\newcommand{\parent}[1]{\left( #1 \right)}

\newcommand{\bbZ}{\mathbb{Z}} 
\newcommand{\bbR}{\mathbb{R}} 

\newcommand{\calF}{\mathcal{F}}

\newcommand{\fphi}{\widetilde{\varphi}} 
\newcommand{\eval}{\widetilde{\epsilon}}
\newcommand{\pd}[1]{\frac{\partial}{\partial #1}}

\definecolor{lightgray}{gray}{0.9}

\begin{document}
\title{Positive spoof Lehmer factorizations}
\author{Grant Molnar}
\author{Guntas Singh} \email{guntas.singh@dsstudents.com} \thanks{Dubai Scholars Private School}

\begin{abstract}
We investigate the integer solutions of Diophantine equations related to Lehmer's totient conjecture. We give an algorithm that computes all nontrivial positive spoof Lehmer factorizations with a fixed number of bases $r$, and enumerate all nontrivial positive spoof Lehmer factorizations with 6 or fewer factors.
\end{abstract}

\maketitle
\section{Introduction}\label{Section: Introduction}

Let $\varphi$ denote the Euler totient function, i.e 
\[
\varphi(n)= n \cdot  \prod_{p|n}(1-1/p).
\]
Here and throughout the paper, $p$ denotes a prime number. In the Foreword of \cite{Lehmer1932}, Lehmer conjectured that if $n$ is an integer such that $\varphi(n) \mid n - 1$, then $n$ is prime. This problem has been attacked repeatedly in the ensuing century, with a strong emphasis on computation aided by a judicious helping of elementary number theory \cite{burcsi2011computational, Cohen-Hagis, Hagis}. These results show that if $n$ is a counterexample to Lehmer's totient conjecture then $n > 10^{20}$ and $n$ has at least 14 distinct factors \cite{Cohen-Hagis}, and that if $3 \mid n$ then $n > 10^{36 \cdot 10^7}$ and $n$ has at least $4 \cdot 10^7$ factors \cite{burcsi2011computational}. Analytic arguments prove that the count of counterexamples to Lehmer's totient conjecture that are bounded by $X$ is subquadratic in $X$; more precisely, it is proven in \cite{Luca-Pomerance} that this quantity is no greater than 
\[
X^{1/2}/ \parent{\log x)^{1/2 + o(1)}}.
\]
Of course, we expect much more to be true.

Lehmer's conjecture is one of several regarding the existence or nonexistence of natural numbers satisfying certain arithmetic properties. The two-millenia-old odd perfect number conjecture is another such claim: it asserts that there are no odd numbers $n$ for which $\sigma(n) = 2n$, where $\sigma(n) = \sum_{d \mid n} d$. The constraints of this problem were relaxed in \cite{Dittmer} and later \cite{BYU-CNT}, by asking for possibly-composite factorizations $x_1^{a_1} \dots x_r^{a_r}$ of $n$ such that ``if'' each $x_i$ were prime, the sum of the factors of $n$ would be $2n$. These factorizations were termed ``spoof factorizations'', and numerous examples of odd spoof perfect factorizations were furnished in \cite{BYU-CNT}. These examples demonstrate the fundamentally arithmetic nature of the odd perfect number conjecture: primality and positivity are necessary ingredients in any putative proof of the conjecture.

In this paper, we adapt the methods of \cite{BYU-CNT, Dittmer} to attack Lehmer's conjecture. Thus, we search for solutions to Diophantine equations of the form
\begin{equation}
k \cdot \prod_{j=1}^{r} (x_{j}-1) = \prod_{j=1}^{r} x_{j} - 1, \ k \in \mathbb{Z}_{>0}, \ x_{j} \in \mathbb{Z}_{>0} \ \text{for} \ 1 \leq j \leq n. \label{Equation: Diophantine equation}
\end{equation}
If we insist that $x_1, \dots, x_r$ are distinct primes, and write $n = x_1 \dots x_r$, then \Cref{Equation: Diophantine equation} becomes $k \cdot \varphi(n) = n - 1$, which is Equation (1) of \cite{Lehmer1932}, so characterizing the integral solutions to \Cref{Equation: Diophantine equation} generalizes Lehmer's totient conjecture. In this paper, we let $x_1, \dots, x_r$ vary over the positive integers; this relaxation from prime factorizations to ``spoof factorizations'' (see \Cref{Definition: spoof factorization}) yields numerous solutions to \Cref{Equation: Diophantine equation}. We retain several counterexamples under the stronger assumption that $x_1, \dots, x_r$ are distinct and odd, as they would be for any genuine counterexample to Lehmer's totient conjecture. Like the odd perfect number conjecture, any proof of Lehmer's totient conjecture must therefore utilize the primality of $x_1, \dots, x_r$, rather than depending only on the algebraic structure that Lehmer's totient problem imputes. That is, any proof of Lehmer's conjecture must assume more than we do in studying \Cref{Equation: Diophantine equation}.

We now state our main theorems.

\begin{theorem}\label{Intro Theorem 1}
    For any integer $r \geq 2$, \Cref{Equation: Diophantine equation} has only finitely many integer solutions with $x_1, \dots, x_r \geq 2$, and these solutions can be explicitly computed.
\end{theorem}

The runtime of the algorithm implicit in \Cref{Intro Theorem 1} depends on ineffective bounds (see Case 3 of the proof of \Cref{Theorem: Algorithm to compute nontrivial Lehmer factorizations}); its time complexity is therefore unknown. However, \Cref{Example: difference of squares}\ref{Subexample: q = 2^ell} with $\ell = 0$ shows that our algorithm has asymptotic runtime no better than $\Omega\parent{2^{2^{r-1}}}$. Executing our algorithm for small $r$, we obtain the following result.

\begin{theorem}\label{Intro Theorem 2}
    For $2 \leq r \leq 6$ an integer, \Cref{Equation: Diophantine equation} has a total of 76 solutions, 31 with odd $x_1, \dots, x_r \geq 3$, and 45 with even $x_1, \dots, x_r \geq 2$.
\end{theorem}

The remainder of our paper is organized as follows. In \Cref{Section: Preliminary definitions}, we establish the definitions necessary to articulate our main results. In \Cref{Section: Preliminary results}, we adapt a couple results from \cite{Lehmer1932}, then proceed to prove a few analytic claims which will be necessary to the development of our algorithm. In \Cref{Section: Our main result}, we prove (a slight generalization of) \Cref{Intro Theorem 1} and \Cref{Intro Theorem 2}, showcase several families of spoof Lehmer factorizations (see \Cref{Example: Powers of 2}, \Cref{Example: Powers of 2 doubled}, and \Cref{Example: difference of squares}), and list those sporadic Lehmer factorizations with 6 or fewer factors that do not lie in these families. Finally, in \Cref{Section: Future work}, we outline a few avenues for future work.

\subsection*{Statements and Declarations}

The authors declare that we have no competing interests.

The code and data produced for this paper are available at \href{https://github.com/grantmolnar/Spoof-Lehmer-factorizations}{this GitHub repository}. The repository includes all the necessary code to implement \Cref{Theorem: Algorithm to compute nontrivial Lehmer factorizations} and reproduce \Cref{Theorem from computation} of this paper, along with the data produced by the code. Additional information can be requested from the corresponding authors if needed.

\subsection*{Acknowledgements}

We thank Jonathon Hales and our reviewer for their careful readings of our paper and for their excellent feedback.

\section{Preliminary definitions}\label{Section: Preliminary definitions}

In this section, we lay out the definitions necessary to state and prove our main results.

\begin{definition}[{\cite{Dittmer}}]\label{Definition: spoof factorization}
A \defi{spoof factorization} $F$ is a finite multiset of ordered pairs of positive integers $\set{(x_1, a_1), \dots, (x_r, a_r)}$. We refer to a pair $(x_i, a_i)$ occurring in a spoof factorization as a \defi{spoof factor}; we refer to $x_i$ as a \defi{spoof base} and $a_i$ as a \defi{spoof exponent} of $F$. We write $\calF$ for the set of all spoof factorizations. In a minor abuse of notation, and following \cite{BYU-CNT}, we write $x_1^{a_1} \dots x_r^{a_r}$ for the spoof factorization $\set{(x_1, a_1), \dots, (x_r, a_r)}$, in analogy with the unique prime factorization $p_1^{a_1} \dots p_r^{a_r}$ for an integer $n \in \mathbb{Z}_{>0}$.
\end{definition}

We interpret a spoof factorization as a factorization $n = x_1^{a_1} \dots x_r^{a_r}$ into powers of not-necessarily-prime factors $x_1, \dots, x_r$, which are nevertheless treated as indivisible and primitive by our ``spoof'' arithmetic functions (see \Cref{Definition: evaluation function} and \Cref{Definition: spoof totient function} below). The bases of the factorization are thus ``spoof primes.'' Note that both bases and exponents are allowed to be equal for distinct factors: for example, $\set{(3, 1), (3, 1)} = 3^1 \cdot 3^1$ is a spoof factorization with two factors. The spoof factorizations of \Cref{Definition: spoof factorization} are somewhat broader than those showcased in \Cref{Equation: Diophantine equation}, because we do not wish to prematurely restrict our domain of study; however, \Cref{Theorem: All exponents are 1} below shows that the only nontrivial spoof Lehmer factorizations are those that satisfy \Cref{Equation: Diophantine equation}.

For brevity and legibility, we drop the appellation ``spoof'' throughout the rest of the paper whenever convenient: all factorizations are taken to be spoof factorizations unless expressly noted otherwise.

We now lay out several conventions and definitions, in analogy to ordinary prime factorizations, which we will employ throughout the rest of the paper.

\begin{definition}\label{Definition: evaluation function}
We define the \defi{(spoof) evaluation function} $\eval:\calF \mapsto \mathbb{Z}$ by
\[
\eval : \set{(x_1, a_1), \dots, (x_r, a_r)} \mapsto \prod_{i=1}^{r} x_{i} ^ {a_{i}}.
\]
\end{definition}

The evaluation function $\eval$ is conceptually analogous to taking the prime factorization of an integer $n$ and reconstructing $n$ from that factorization. We acknowledge the risk of confusion between $x_1^{a_1} \dots x_r^{a_r}$ as a shorthand for an element of $\calF$, and as a product in $\bbZ$. Such subtleties cause little trouble in classical number theory. We hope that context will make matters clear here as well.

\begin{definition}\label{Definition: spoof totient function} 
We define the \defi{(spoof) totient function} $\fphi : \calF \to \bbZ$ by 
\[
\fphi : x_1^{a_1} \dots x_r^{a_r} \mapsto (x_1^{a_1} - x_1^{a_1 - 1}) \cdots (x_r^{a_r} - x_r^{a_r - 1}) = \eval(x_1^{a_1} \dots x_r^{a_r}) \cdot \prod_{i=1}^{r} \parent{1-\frac{1}{x_{i}}}.
\]
\end{definition}
   
The spoof totient function $\fphi$ is given by direct analogy with the Euler totient function. Indeed, if $x_1, \dots, x_r$ are distinct primes, then $\varphi(x_1^{a_1} \dots x_r^{a_r}) = \fphi(x_1^{a_1} \dots x_r^{a_r})$.

At last, we come to our main subject of interest, namely spoof analogues to natural numbers satisfying Equation (1) of \cite{Lehmer1932}.

\begin{definition}\label{Definition: k-Lehmer factorizations}
    Let $k \in \bbZ$. We say that a (spoof) factorization $F \in \calF$ is a \defi{$k$-Lehmer factorization} if the following conditions hold:
    \begin{itemize}
        \item We have 
        \begin{equation}
            k \cdot \fphi(F) = \eval(F) - 1; \label{Equation: Lehmer equation}
        \end{equation}\label{Condition: F is Lehmer}
        \item $F$ has at least two factors, or at least one factor with an exponent greater than 1. \label{Condition: F is composite}
    \end{itemize}

\end{definition}

We recover \Cref{Equation: Diophantine equation} from \Cref{Equation: Lehmer equation} under the assumption that the exponents $a_1, \dots, a_r$ of $F = x_1^{a_1} \dots x_r^{a_r}$ are all 1. The first condition of \Cref{Definition: k-Lehmer factorizations} is given by analogy with Equation (1) of \cite{Lehmer1932}, and the second condition restricts our attention to composite factorizations.

\begin{definition}\label{Definition: Lehmer factorizations}
We say that $F \in \calF$ is a \defi{Lehmer factorization} if it is a $k$-Lehmer factorization for some $k \in \bbZ$.
\end{definition}

We refer to \Cref{Equation: Lehmer equation} as the \defi{Lehmer equation}, or the \defi{$k$-Lehmer equation} if we view $k$ as fixed.

\begin{example}\label{Example: All the factors are 1}
	Let $a_{1}, a_{2}, ... , a_{r} $ be positive integers. We compute
	\[
	\fphi(1^{a_{1}} 1^{a_{2}} \dots 1^{a_{r}}) = 0 = \eval(1^{a_{1}} 1^{a_{2}} \dots 1^{a_{r}}) - 1,
	\]
    so the spoof factorization $1^{a_{1}} 1^{a_{2}} \dots 1^{a_{r}} \in \calF $ is a $k$-Lehmer factorization for any $k \in \bbZ$.
\end{example}

\begin{definition}\label{Definition: Trivial Lehmer factorization}
	If every base of $F \in \calF$ is $1$, we say that $F$ is \defi{trivial}. Otherwise, $F$ is a \defi{nontrivial} factorization.
\end{definition}

\begin{proposition}
	If $F$ is a Lehmer factorization and $1$ is a base of $F$, then $F$ is a trivial factorization.
\end{proposition}

\begin{proof}
	Let $F$ be a Lehmer factorization. If $1$ is a base of $F$, then $\fphi(F) = 0$. Now, as $F$ is a Lehmer factorization, we have $k \cdot 0 = 0 = \eval(F) - 1$ for some $k \in \bbZ$, and thus $\eval(F) = 1$. This implies that  $1$ is the only integer occurring as a base of $F$, as desired.
\end{proof}

We close this section with one more definition, which will be useful in the remainder of the paper.

\begin{definition}\label{Definition: odd factorizations}
A (spoof) factorization $x_1^{a_1} \dots x_r^{a_r}$ is \defi{odd} if every $x_i$ is odd, and \defi{even} otherwise.
\end{definition}

\section{Preliminary results}\label{Section: Preliminary results}

With the definitions of \Cref{Section: Preliminary definitions} in-hand, we are now ready to prove several results which we shall require on the road to our main theorem. The first two theorems are spoof analogues of Theorems 1 and 2 of \cite{Lehmer1932}.

\begin{theorem}\label{Theorem: All exponents are 1}
    If $F = x_1^{a_1} \dots x_r^{a_r} \in \calF$ is a nontrivial spoof Lehmer factorization, then $a_1 = \dots = a_r = 1$.
\end{theorem}

\begin{proof}
By \Cref{Equation: Lehmer equation} and the Euclidean algorithm, we have
    \[
     \gcd(\fphi(F) , \eval(F)) = 1 .
    \]
    If $a_m > 1$ for some $m \in \set{1, \dots, r}$, then $x_{m}^{a_{m}} - x_{m} ^{a_{m} - 1} \equiv 0 \pmod{x_{m}}$ is a factor of $\fphi(F)$ on the left, while we have $\eval(F) - 1 \equiv -1 \pmod{x_{m}}$ on the right, which is a contradiction because $x_m > 1$.
\end{proof}

\begin{remark}
    The proof of \Cref{Theorem: All exponents are 1} holds without modification if we allow nonpositive $x_1, \dots, x_r$, so long as $x_1, \dots, x_r \neq -1$.
\end{remark}

For ease of notation, we write $x_1 \dots x_r$ for the spoof factorization $x_1^1 \dots x_r^1$.

\begin{theorem}\label{Theorem: factors of eval are not factors plus 1 of bases}
    Let $F = x_1 \dots x_r \in \calF$ be a nontrivial spoof Lehmer factorization. For any $n \in \mathbb{Z}$ a factor of $\eval(F)$, and any base $x_i$ of $F$, $x_i \not\equiv 1 \pmod n$. 
\end{theorem}

\begin{proof}
    Suppose by way of contradiction that $x_i \equiv 1 \pmod n$. Then $\fphi(x_i^1) \equiv 0 \pmod n$; considering the Lehmer equation \Cref{Equation: Lehmer equation} modulo $n$, we obtain
    \[
	0 \equiv k \cdot \fphi(F) = \eval(F) - 1 \equiv -1 \pmod n. 
    \]
	This is a contradiction.
\end{proof}

\begin{corollary}\label{Corollary: No bases of the for m x + 1}
	If $F = x_1 \dots x_r \in \calF$ is a nontrivial spoof Lehmer factorization, then for $1 \leq i, j \leq n$, we have $x_i \not\equiv 1 \pmod {x_j}$.
\end{corollary}

\begin{proof}
	Since $x_j$ is a base of $F$, $x_j$ is a factor of $\eval(F)$. The claim is now immediate from \Cref{Theorem: factors of eval are not factors plus 1 of bases}.
\end{proof}

\begin{corollary}\label{Corollary: all even or all odd}
	If $F = x_1 \dots x_r \in \calF$ is a nontrivial spoof Lehmer factorization, and $p$ is a prime dividing $\eval(F)$, then for $1 \leq i \leq r$, we have $x_i \not\equiv 1 \pmod p$. In particular, if $F$ is even, then all of its bases are even.
\end{corollary}

\begin{proof}
	The first half of the statement follows by letting $m = p$ prime in \Cref{Theorem: factors of eval are not factors plus 1 of bases}. The second half follows by letting $p = 2$.
\end{proof}

Suppose that $F = x_1 \dots x_r \in \calF$ is a nontrivial spoof $k$-Lehmer factorization. Rearranging \Cref{Equation: Lehmer equation}, we obtain
\[
k = \frac{\eval(F) - 1}{\fphi(F)} = \frac{x_1 \cdots x_r - 1}{(x_1 - 1) \cdots (x_r - 1)}.
\]
This motivates the following definition.

\begin{definition}\label{Definition: kappa}
    For $r$ a positive integer, we define $\kappa_r : [2, \infty)^r \to \bbR$ by 
    \[
    \kappa_r: (t_1, \dots, t_r) \mapsto \frac{t_1 \cdots t_r - 1}{(t_1 - 1) \cdots (t_r - 1)}
    \]
\end{definition}

\begin{theorem}\label{Theorem: kappa is decreasing}
    The function $\kappa_r : [2, \infty)^r \to \bbR$ is decreasing in each of its arguments, and 
    \begin{equation}
    1 < \kappa_r(t_1, \dots, t_r) \leq 2^r - 1\label{Equation: inequalities for kappa}
    \end{equation}
    for all $(t_1, \dots, t_r) \in [2, \infty)^r$.
\end{theorem}

\begin{proof}
    For any $j \in \set{1, \dots, r}$, we compute
    \[
    \pd {t_j} \kappa_{r}(t_1, \dots, t_r) = - \frac{1}{(t_j-1)^2} \cdot \kappa_{r-1}(t_1, \dots, t_{j-1}, t_{j+1}, \dots, t_r) < 0,
    \]
    which implies that $\kappa_r$ is strictly decreasing in each of its arguments as desired. Now \Cref{Equation: inequalities for kappa} follows by evaluating $\kappa_r(2, \dots, 2) = 2^r - 1$ and $\lim_{t\to\infty} \kappa_r(t, \dots, t) = 1$.
\end{proof}



We suppress the dependence of $\kappa_r$ on $r$ when no confusion arises. We utilize $\kappa$ to define two auxiliary functions.

\begin{definition}\label{Definition: L and U}
    Let $2 \leq t_1 \leq t_2 \leq \dots \leq t_s$ be real numbers, and let $r \geq s$ be an integer. We define 
    \[
    L(t_1, \dots, t_s; r) \coloneqq \inf \set{\kappa(t_1, \dots, t_s, u_{s+1}, \dots, u_r) \ : \ t_s \leq u_{s+1} \leq \dots \leq u_r}
    \]
    and
    \[
    U(t_1, \dots, t_s; r) \coloneqq \sup \set{\kappa(t_1, \dots, t_s, u_{s+1}, \dots, u_r) \ : \ t_s \leq u_{s+1} \leq \dots \leq u_r}.
    \]
\end{definition}

$L$ and $U$ provide lower and upper bounds on the function $\kappa$ when some of its arguments are prescribed. \Cref{Theorem: kappa is decreasing} tells us that $\kappa$ is bounded, so the completeness of the real numbers ensures that $L$ and $U$ exist.

It is immediate from \Cref{Definition: L and U} that
\begin{equation}
L(t_1, \dots, t_s; r) \leq L(t_1, \dots, t_{s+1}; r) \leq U(t_1, \dots, t_{s+1}; r) \leq U(t_1, \dots, t_s; r)\label{Equation: extending L and U push them toward each other}
\end{equation}
for any $s \leq r$ and any $2 \leq t_1 \leq t_2 \leq \dots \leq t_s \leq t_{s+1}$. 

In the following lemma we show how $L$ and $U$ can easily be computed for any nonnegative integers $s \leq r$ and real numbers $2 \leq t_1 \leq t_2 \leq \dots \leq t_s$.

\begin{lemma}\label{Lemma: Evaluations of L and U}
    We have
    \[
    L(\emptyset; r) = 1
    \]
    and
    \[
    U(\emptyset; r) = 2^r - 1.
    \]
    Let $2 \leq t_1 \leq \dots \leq t_s$ be real numbers with $s > 0$, and let $r \geq s$ be an integer. We have
    \[
    L(t_1, \dots, t_s; r) = \frac{t_1 \cdots t_s}{(t_1 - 1) \cdots (t_s - 1)},
    \]
    and
    \[
    U(t_1, \dots, t_s; r) = \frac{t_1 \cdots t_{s-1} t_{s}^{r-s+1} - 1}{(t_{1} - 1) \cdots (t_{s-1} - 1) (t_{s} - 1)^{r-s+1}}.
    \]
\end{lemma}

\begin{proof}
    By \Cref{Theorem: kappa is decreasing}, we see
	\[
    L(t_1, \dots, t_s; r) = \lim_{u \to \infty} \kappa(t_1, \dots, t_s, u, \dots, u),
    \]
    which gives us both claims about $L$. Similarly, we have
    \[
    U(\emptyset; r) = \kappa(2, \dots, 2) = 2^r - 1,
    \]
    and
    \begin{equation}
    U(t_1, \dots, t_s; r) = \kappa(t_1, \dots, t_s, \dots, t_s),\label{Equation: U in terms of kappa}
    \end{equation}
    where $t_s$ is repeated $r - s + 1$ times as an argument of $\kappa$, which gives us both claims about $U$.
\end{proof}

We now verify that $L$ and $U$, like $\kappa$, are decreasing in their arguments with respect to a fixed $r$.

\begin{lemma}\label{Lemma: L and U are decreasing} 
    Given integers $0 \leq s \leq r$, $L(t_1, \dots, t_s; r) $ and $U(t_1, \dots, t_s; r) $ are monotonically decreasing in $t_1, \dots, t_s$. 
\end{lemma}

\begin{proof}
    The monotonicity of $U$ inherits from that of $\kappa$, by \Cref{Equation: U in terms of kappa}. It thus suffices to prove that $L$ is monotonically decreasing. Indeed, for $1 \leq i \leq s$ we have
    \begin{align*}
	\pd {t_i} L(t_1, \dots, t_s; r) = -\frac{1}{t_i(t_i - 1)} L(t_1, \dots, t_s; r) < 0,
    \end{align*}
	and our claim follows.
\end{proof}

\begin{definition}
In a minor abuse of notation, if $F = x_1 \dots x_s$ and $s \leq r$, we write $L(F; r) \coloneqq L(x_1, \dots, x_s; r)$, and likewise write $U(F; r) \coloneqq U(x_1, \dots, x_s; r)$.
\end{definition}

\section{Our main result}\label{Section: Our main result}

In this section, we prove \Cref{Intro Theorem 1} and \Cref{Intro Theorem 2}, and describe several families of spoof Lehmer factorizations.

\begin{definition}\label{Definition: Spoof factorization extension}
    Let $F = x_1 \dots x_r$ and $G = y_1 \dots y_s$ be spoof factorizations whose spoof exponents are all 1, written so that $x_1 \leq x_2 \leq \dots \leq x_r$ and $y_1 \leq y_2 \leq \dots \leq y_s$. We say that $F$ \defi{extends} $G$ if $s \leq r$ and $x_i = y_i$ for $1 \leq i \leq s$. If $F$ extends $G$, we say $G$ is a \defi{partial factorization} of $F$.
\end{definition}

Taking inspiration from Theorem 6.2 of \cite{BYU-CNT}, we construct an algorithm to find all nontrivial spoof Lehmer factorizations with $r$ factors. It suffices to devise an algorithm to find all nontrivial spoof $k$-Lehmer factorization with $r$ factors, as the following lemma shows.

\begin{lemma}\label{Lemma: Only finitely many values for k given r}
    Suppose $F \in \calF$ is a nontrivial spoof $k$-Lehmer factorization with $r$ factors. Then for any partial factorization $G$ of $F$ we have
    \begin{equation}
     L(G; r) \leq k \leq U(G; r),\label{Equation: Core inequality}
    \end{equation}
    with equality on the left if and only if $F = G$. In particular, we have
    \begin{equation}
    1 < k \leq 2^{r} - 1.\label{Equation: bounds on k when s = 0}
    \end{equation}
\end{lemma}

\begin{proof}
    The inequality \Cref{Equation: Core inequality} is immediate from \Cref{Definition: L and U}, and it is strict on the left for $s < r$ because $L$ is obtained as the limit of a strictly decreasing function. The inequality \Cref{Equation: bounds on k when s = 0} follows from \Cref{Equation: Core inequality} and from \Cref{Lemma: Evaluations of L and U} by letting $G$ be the empty factorization.
\end{proof}

If $F \in \calF$ is a nontrivial \emph{odd} spoof $k$-Lehmer factorization, our bound on $k$ may be improved to 
\begin{equation}
1 < k \leq \frac{3^r - 1}{2^r}.\label{Equation: improved bound in odd case}
\end{equation}

We now recapitulate \Cref{Intro Theorem 1} in the language we have established. In fact, \Cref{Theorem: Algorithm to compute nontrivial Lehmer factorizations} is slightly stronger than \Cref{Intro Theorem 1}, but the two are equivalent in the presence of \Cref{Theorem: All exponents are 1}.
    
\begin{theorem}\label{Theorem: Algorithm to compute nontrivial Lehmer factorizations}
    Let $r$ be a positive integer. There exist only finitely many nontrivial Lehmer factorizations with $r$ bases. Moreover, these Lehmer factorizations can be explicitly and verifiably computed in finite time.
\end{theorem}

\begin{proof}
    Let $r$ be a positive integer. It suffices to prove that there are only finitely many nontrivial $k$-Lehmer factorizations with $r$ bases, since by \Cref{Lemma: Only finitely many values for k given r} we can iterate $k$ from $2$ to $2^{r} - 1$. We now present an algorithm which generates all such $k$-Lehmer factorizations.

    We proceed inductively, beginning with the empty factorization $P_0 = \emptyset$. By \Cref{Definition: L and U}, if $P_i = x_1 \dots x_i$ satisfies \Cref{Equation: Core inequality} with $s = i$, then $P_{i - 1} = x_1 \dots x_{i-1}$ satisfies \Cref{Equation: Core inequality} with $s = i - 1$. Also, by \Cref{Lemma: Only finitely many values for k given r}, $P_r$ satisfies \Cref{Equation: Core inequality} if and only if $P_r$ is a $k$-Lehmer factorization. It remains only to show that given $P_{i-1}$ satisfying \Cref{Equation: Core inequality} with $s = i - 1$, we can construct all $P_i$ extending $P_{i-1}$ that satisfy \Cref{Equation: Core inequality} with $s = i$.

	Fix $P_{i-1} = x_1 \dots x_{i-1}$ satisfying \Cref{Equation: Core inequality}, and consider the sequence $P_i(x) \coloneqq P_{i-1} \cdot x$ for $x = x_{i-1}, x_{i-1} + 1, \dots$ (or $x = 2, 3, \dots$ if $i = 1$). We discard any such $P_i(x)$ for which $L(P_i(x); r) \geq k$, in keeping with \Cref{Equation: Core inequality}. By \Cref{Lemma: Evaluations of L and U}, we see
	\[
	\lim_{x \to \infty} L(x_1, \dots, x_{i-1}, x; r) = \lim_{x \to \infty} U(x_1, \dots, x_{i-1}, x; r) = L(x_1, \dots, x_{i-1}; r).
	\]
	We now have three cases.

	\textbf{Case 1: $L(x_1, \dots, x_{i-1}; r) > k$}. This case is impossible by the induction hypothesis.
	
	\textbf{Case 2: $L(x_1, \dots, x_{i-1}; r) = k$}. In this case, \Cref{Lemma: Only finitely many values for k given r} tells us $i > r$, which is absurd.
	
	\textbf{Case 3: $L(x_1, \dots, x_{i-1}; r) < k$}. In this case, \Cref{Equation: Core inequality} is violated for all $x$ sufficiently large. Moreover, \Cref{Lemma: L and U are decreasing} tells us that $U$ is decreasing in its arguments, so $U(P_i(x); r) < k$ implies $U(P_i(y); r) < k$ for all $y \geq x$, and we can stop incrementing $x$ as soon as $U(P_i(x); r) < k$ obtains.

	Our claim follows. 
\end{proof}

\begin{remark}\label{Remark: Algorithm improvements}
	Several improvements to the algorithm outlined in \Cref{Theorem: Algorithm to compute nontrivial Lehmer factorizations} are possible, if the goal is to enumerate all nontrivial Lehmer factorizations. By \Cref{Corollary: all even or all odd}, we can either assume all bases of $P$ are odd or all are even, and in either case increment $x$ by $2$ instead of by $1$. If we are in the odd case, then \Cref{Equation: improved bound in odd case} lets us restrict the range over which $k$ varies. Regardless, \Cref{Corollary: No bases of the for m x + 1} lets us skip some bases in our induction step. In fact, \Cref{Theorem: factors of eval are not factors plus 1 of bases} enables further, nontrivial refinements to our algorithm, which we decline to pursue.
\end{remark}

Estimating the time complexity of the algorithm implicit in \Cref{Theorem: Algorithm to compute nontrivial Lehmer factorizations} would give an upper bound on the value of $x$ for which \Cref{Equation: Core inequality} first fails in Case 3 of the theorem, as a function of $r$ and $x_1, \dots, x_{i-1}$.

In addition to \Cref{Example: All the factors are 1} above, we can construct several nontrivial infinite families of spoof Lehmer factorizations by hand. 

Write $[a]^s$ for the spoof factorization $a \dots a$ with $a$ repeated $s$ times. 

\begin{example}\label{Example: Powers of 2 doubled}
        For $s \geq 1$ an integer, the spoof factorization $[2]^{s} \cdot 2^{s}$ is a $(2^{s} + 1)$-Lehmer factorization. 
\end{example}

\begin{example}\label{Example: Powers of 2}
	For $s \geq 2$ an integer, the spoof factorization $[2]^s$ is a $(2^{s}-1)$-Lehmer factorization. More generally, let $F$ be an even Lehmer factorization or an even factorization of the form $F = x^1$, and let $o$ be the smallest positive integer such that $2^{o} \equiv 1 \pmod {\fphi(F)}$; we know $o$ exists because $\fphi(F)$ is odd. For $s \geq 1$ an integer, $[2]^{os} \cdot F$ is an even Lehmer factorization. For instance, if $F = [2]^2 \cdot 4$, we obtain $[2]^4 \cdot 4, [2]^6 \cdot 4, \dots$, and if $F = 6$, we obtain $[2]^4 \cdot 6, [2]^8 \cdot 6, \dots$.
 \end{example}

\begin{example}\label{Example: difference of squares}
    Let $\set{a_1, \dots, a_m}, k, q$ be positive integers such that
    \[
    \prod_{i = 1}^m a_i = q - 1 \ \text{and} \ k \cdot \prod_{i = 1}^m (a_i - 1) = q.
    \]
    Then for any $s \geq 0$, 
    \[
    \parent{q^{2^s} - 1} \cdot \prod_{j = 0}^{s - 1} \parent{q^{2^j} + 1} \cdot \prod_{i = 1}^m a_i 
    \]
    is a $k$-Lehmer factorization. The verification is straightforward, and follows from two elementary identities, namely $x^{2^s} - 1 = (x - 1) (x + 1) (x^2 + 1) \cdots (x^{2^{s-1}} + 1)$ and $(x - 1)^2 - 1 = x \cdot (x - 2)$. We now record several instances of this phenomenon:
    \begin{enumerate}[label=(\roman*)]
        \item for $\ell \geq 0$, $[2]^\ell \cdot \parent{(2^\ell + 1)^{2^{s}} - 1} \cdot \prod_{i = 0}^{s-1} \parent{(2^\ell + 1)^{2^i} + 1}$, with $k =2^\ell + 1$;\label{Subexample: q = 2^ell}
        \item $[2]^2 \cdot \parent{5^{2^{s}} - 1} \cdot \prod_{i = 0}^{s-1} \parent{5^{2^i} + 1}$, with $k = 5$;\label{Subexample: q = 5}
        \item $[2]^4 \cdot \parent{17^{2^s} - 1} \cdot \prod_{i = 0}^{s-1} \parent{17^{2^i} + 1}$, with $k = 17$;\label{Subexample: q = 17}
        \item $[2]^3 \cdot 4 \cdot \parent{33^{2^s} - 1} \cdot \prod_{i = 0}^{s-1} \parent{33^{2^i} + 1}$, with $k = 11$;\label{Subexample: q = 33}
        \item $[2]^5 \cdot 4 \cdot \parent{129^{2^s} - 1} \cdot \prod_{i = 0}^{s-1} \parent{129^{2^i} + 1}$, with $k = 43$; \label{Subexample: q = 129}
        \item $[2]^3 \cdot 6 \cdot 8 \cdot \parent{385^{2^s} - 1} \cdot \prod_{i = 0}^{s-1} \parent{385^{2^i} + 1}$, with $k = 11$;\label{Subexample: q = 385}
        \item $[2]^2 \cdot 8 \cdot 12 \cdot \parent{385^{2^s} - 1} \cdot \prod_{i = 0}^{s-1} \parent{385^{2^i} + 1}$, with $k = 5$;\label{Subexample: q = 385 2}
        \item $[5]^3 \cdot 43 \cdot \parent{5376^{2^s} - 1} \cdot \prod_{i = 0}^{s-1} \parent{5376^{2^i} + 1}$, with $k = 2$;\label{Subexample: q = 5376}
        \item $2 \cdot [8]^3 \cdot 206 \cdot \parent{210945^{2^s} - 1} \cdot \prod_{i = 0}^{s-1} \parent{210945^{2^i} + 1}$, with $k = 3$.\label{Subexample: q = 210945}
    \end{enumerate}
    Note that if $\set{a_1, \dots, a_m}, q$ describe a family of Lehmer factorizations in the manner described, then $\set{a_1, \dots, a_m, q + 1}, k, q^2$ describes the same family with shifted index.
\end{example}

 \begin{example}\label{Example: Fermat numbers product}
        For $s \geq 0$, write $T(s) \coloneqq 2^{2^{s+1}} - 2^{2^{s}} - 1$. For any factorization $T(s) = d_1 d_2$, the spoof factorization $(2^{2^{s}} + d_1) \cdot (2^{2^{s}} + d_2) \cdot \prod_{i = 0}^{s -1} (2^{2^i} + 1)$ is a $2$-Lehmer factorization.
\end{example}

These families were discovered by extrapolating from spoof Lehmer factorizations computed for \Cref{Theorem from computation} below.

Python code that implements our algorithm described in the proof of \Cref{Theorem: Algorithm to compute nontrivial Lehmer factorizations}, with the refinements outlined in \Cref{Remark: Algorithm improvements}, can be found at \href{https://github.com/grantmolnar/Spoof-Lehmer-factorizations}{this link}. Running this code, we obtain \Cref{Intro Theorem 2}, which we restate here in slightly greater detail.

\begin{theorem}[\Cref{Intro Theorem 2}]\label{Theorem from computation}
    There are 31 nontrivial odd spoof Lehmer factorizations with 6 or fewer factors; 11 of these are instances of the families outlined in \Cref{Example: difference of squares} and \Cref{Example: Fermat numbers product}, and the remaining $20$ ``sporadic'' examples are marked with check-marks in \Cref{Table: Odd Lehmer factorizations}. There are 45 nontrivial even spoof Lehmer factorizations with 6 or fewer factors; $27$ of these are instances of the families outlined in \Cref{Example: Powers of 2 doubled}, \Cref{Example: Powers of 2}, and \Cref{Example: difference of squares}, and the remaining $18$ ``sporadic'' examples are marked with check-marks  in \Cref{Table: Even Lehmer factorizations}.
\end{theorem}

\Cref{Theorem from computation} was proved computationally by running the algorithm described in the proof of \Cref{Theorem: Algorithm to compute nontrivial Lehmer factorizations} for 6.5 CPU-hours
in the odd case, and 0.4 CPU-hours in the even case. 

\begin{table}[ht]
\centering
\caption{Odd Lehmer factorizations with at most 6 factors}
\label{Table: Odd Lehmer factorizations}
\begin{tabular}{>{\columncolor{white}}c c c c}
$r$ & $k$ & Lehmer factorization & Sporadic? \\
\rowcolor{lightgray}
2 & 2 & $3 \cdot 3$ &
\Cref{Example: difference of squares}\ref{Subexample: q = 2^ell},
\Cref{Example: Fermat numbers product} \\
3 & 2 & $3 \cdot 5 \cdot 15$ &
\Cref{Example: difference of squares}\ref{Subexample: q = 2^ell},
\Cref{Example: Fermat numbers product} \\
\rowcolor{lightgray}
4 & 2 & $3 \cdot 5 \cdot 17 \cdot 255$ &
\Cref{Example: difference of squares}\ref{Subexample: q = 2^ell},
\Cref{Example: Fermat numbers product} \\
4 & 5 & $3 \cdot 3 \cdot 3 \cdot 3$ &
$\checkmark$ \\
\rowcolor{lightgray}
5 & 2 & $3 \cdot 5 \cdot 17 \cdot 257 \cdot 65535$ &
\Cref{Example: difference of squares}\ref{Subexample: q = 2^ell},
\Cref{Example: Fermat numbers product} \\
5 & 2 & $3 \cdot 5 \cdot 17 \cdot 285 \cdot 2507$ &
\Cref{Example: Fermat numbers product} \\
\rowcolor{lightgray}
5 & 2 & $5 \cdot 5 \cdot 5 \cdot 43 \cdot 5375$ &
\Cref{Example: difference of squares}\ref{Subexample: q = 5376} \\
5 & 2 & $5 \cdot 5 \cdot 9 \cdot 9 \cdot 89$ &
$\checkmark$ \\
\rowcolor{lightgray}
6 & 2 & $3 \cdot 5 \cdot 17 \cdot 257 \cdot 65537 \cdot 4294967295$ &
\Cref{Example: difference of squares}\ref{Subexample: q = 2^ell},
\Cref{Example: Fermat numbers product} \\
6 & 2 & $3 \cdot 5 \cdot 17 \cdot 257 \cdot 65555 \cdot 226112997$ &
\Cref{Example: Fermat numbers product} \\
\rowcolor{lightgray}
6 & 2 & $3 \cdot 5 \cdot 17 \cdot 257 \cdot 65717 \cdot 23794275$ &
\Cref{Example: Fermat numbers product} \\
6 & 2 & $3 \cdot 5 \cdot 17 \cdot 257 \cdot 68975 \cdot 1314417$ &
\Cref{Example: Fermat numbers product} \\
\rowcolor{lightgray}
6 & 2 & $3 \cdot 5 \cdot 17 \cdot 353 \cdot 929 \cdot 83623935$ &
$\checkmark$ \\
6 & 2 & $3 \cdot 5 \cdot 17 \cdot 377 \cdot 1217 \cdot 2295$ &
$\checkmark$ \\
\rowcolor{lightgray}
6 & 2 & $3 \cdot 5 \cdot 17 \cdot 395 \cdot 1059 \cdot 2303$ &
$\checkmark$ \\
6 & 2 & $3 \cdot 5 \cdot 33 \cdot 53 \cdot 69 \cdot 8343$ &
$\checkmark$ \\
\rowcolor{lightgray}
6 & 2 & $3 \cdot 9 \cdot 9 \cdot 23 \cdot 131 \cdot 732159$ &
$\checkmark$ \\
6 & 2 & $5 \cdot 5 \cdot 5 \cdot 43 \cdot 5377 \cdot 28901375$ &
\Cref{Example: difference of squares}\ref{Subexample: q = 5376} \\
\rowcolor{lightgray}
6 & 2 & $5 \cdot 5 \cdot 5 \cdot 43 \cdot 5387 \cdot 2632285$ &
$\checkmark$ \\
6 & 2 & $5 \cdot 5 \cdot 5 \cdot 43 \cdot 5407 \cdot 937505$ &
$\checkmark$ \\
\rowcolor{lightgray}
6 & 2 & $5 \cdot 5 \cdot 5 \cdot 43 \cdot 5477 \cdot 291475$ &
$\checkmark$ \\
6 & 2 & $5 \cdot 5 \cdot 5 \cdot 43 \cdot 5717 \cdot 90115$ &
$\checkmark$ \\
\rowcolor{lightgray}
6 & 2 & $5 \cdot 5 \cdot 5 \cdot 43 \cdot 6215 \cdot 39817$ &
$\checkmark$ \\
6 & 2 & $5 \cdot 5 \cdot 5 \cdot 43 \cdot 6487 \cdot 31385$ &
$\checkmark$ \\
\rowcolor{lightgray}
6 & 2 & $5 \cdot 5 \cdot 5 \cdot 43 \cdot 8507 \cdot 14605$ &
$\checkmark$ \\
6 & 2 & $5 \cdot 5 \cdot 5 \cdot 45 \cdot 807 \cdot 267023$ &
$\checkmark$ \\
\rowcolor{lightgray}
6 & 2 & $5 \cdot 5 \cdot 5 \cdot 47 \cdot 453 \cdot 2661375$ &
$\checkmark$ \\
6 & 2 & $5 \cdot 5 \cdot 5 \cdot 47 \cdot 457 \cdot 50659$ &
$\checkmark$ \\
\rowcolor{lightgray}
6 & 2 & $5 \cdot 5 \cdot 5 \cdot 65 \cdot 129 \cdot 2325$ &
$\checkmark$ \\
6 & 2 & $5 \cdot 7 \cdot 7 \cdot 7 \cdot 133 \cdot 228095$ &
$\checkmark$ \\
\rowcolor{lightgray}
6 & 4 & $3 \cdot 3 \cdot 5 \cdot 5 \cdot 9 \cdot 89$ &
$\checkmark$ \\
\end{tabular}
\end{table}

\begin{table}[ht]
\centering
\caption{Even Lehmer factorizations with at most 6 factors}
\label{Table: Even Lehmer factorizations}
\begin{tabular}{>{\columncolor{white}}c c c c}
$r$ & $k$ & Lehmer factorization & Sporadic? \\
\rowcolor{lightgray}
2 & 3 & $2 \cdot 2$ &
\Cref{Example: Powers of 2 doubled}, \Cref{Example: Powers of 2}, \Cref{Example: difference of squares}\ref{Subexample: q = 2^ell} \\
3 & 3 & $2 \cdot 4 \cdot 8$ &
\Cref{Example: difference of squares}\ref{Subexample: q = 2^ell} \\
\rowcolor{lightgray}
3 & 5 & $2 \cdot 2 \cdot 4$ &
\Cref{Example: Powers of 2 doubled}, \Cref{Example: difference of squares}\ref{Subexample: q = 5} \\
3 & 7 & $2 \cdot 2 \cdot 2$ &
\Cref{Example: Powers of 2} \\
\rowcolor{lightgray}
4 & 3 & $2 \cdot 4 \cdot 10 \cdot 80$ &
\Cref{Example: difference of squares}\ref{Subexample: q = 2^ell} \\
4 & 5 & $2 \cdot 2 \cdot 6 \cdot 24$ &
\Cref{Example: difference of squares}\ref{Subexample: q = 5} \\
\rowcolor{lightgray}
4 & 7 & $2 \cdot 2 \cdot 4 \cdot 4$ &
$\checkmark$ \\
4 & 9 & $2 \cdot 2 \cdot 2 \cdot 8$ &
\Cref{Example: Powers of 2 doubled}, \Cref{Example: difference of squares}\ref{Subexample: q = 2^ell} \\
\rowcolor{lightgray}
4 & 15 & $2 \cdot 2 \cdot 2 \cdot 2$ &
\Cref{Example: Powers of 2} \\
5 & 3 & $2 \cdot 4 \cdot 10 \cdot 82 \cdot 6560$ &
\Cref{Example: difference of squares}\ref{Subexample: q = 2^ell} \\
\rowcolor{lightgray}
5 & 3 & $2 \cdot 4 \cdot 10 \cdot 92 \cdot 670$ &
$\checkmark$ \\
5 & 3 & $2 \cdot 4 \cdot 10 \cdot 100 \cdot 422$ &
$\checkmark$ \\
\rowcolor{lightgray}
5 & 3 & $2 \cdot 4 \cdot 10 \cdot 112 \cdot 290$ &
$\checkmark$ \\
5 & 5 & $2 \cdot 2 \cdot 6 \cdot 26 \cdot 624$ &
\Cref{Example: difference of squares}\ref{Subexample: q = 5} \\
\rowcolor{lightgray}
5 & 5 & $2 \cdot 2 \cdot 8 \cdot 12 \cdot 384$ &
\Cref{Example: difference of squares}\ref{Subexample: q = 385 2} \\
5 & 9 & $2 \cdot 2 \cdot 2 \cdot 10 \cdot 80$ &
\Cref{Example: difference of squares}\ref{Subexample: q = 2^ell} \\
\rowcolor{lightgray}
5 & 11 & $2 \cdot 2 \cdot 2 \cdot 4 \cdot 32$ &
\Cref{Example: difference of squares}\ref{Subexample: q = 33} \\
5 & 17 & $2 \cdot 2 \cdot 2 \cdot 2 \cdot 16$ &
\Cref{Example: Powers of 2 doubled}, \Cref{Example: difference of squares}\ref{Subexample: q = 17} \\
\rowcolor{lightgray}
5 & 19 & $2 \cdot 2 \cdot 2 \cdot 2 \cdot 6$ &
\Cref{Example: Powers of 2} \\
5 & 21 & $2 \cdot 2 \cdot 2 \cdot 2 \cdot 4$ &
\Cref{Example: Powers of 2} \\
\rowcolor{lightgray}
5 & 31 & $2 \cdot 2 \cdot 2 \cdot 2 \cdot 2$ &
\Cref{Example: Powers of 2} \\
6 & 3 & $2 \cdot 4 \cdot 10 \cdot 82 \cdot 6562 \cdot 43046720$ &
\Cref{Example: difference of squares}\ref{Subexample: q = 2^ell} \\
\rowcolor{lightgray}
6 & 3 & $2 \cdot 4 \cdot 10 \cdot 82 \cdot 12440 \cdot 13882$ &
$\checkmark$ \\
6 & 3 & $2 \cdot 4 \cdot 10 \cdot 88 \cdot 1120 \cdot 9944$ &
$\checkmark$ \\
\rowcolor{lightgray}
6 & 3 & $2 \cdot 4 \cdot 26 \cdot 26 \cdot 26 \cdot 8272$ &
$\checkmark$ \\
6 & 3 & $2 \cdot 8 \cdot 8 \cdot 8 \cdot 206 \cdot 210944$ &
\Cref{Example: difference of squares}\ref{Subexample: q = 210945} \\
\rowcolor{lightgray}
6 & 5 & $2 \cdot 2 \cdot 6 \cdot 26 \cdot 626 \cdot 390624$ &
\Cref{Example: difference of squares}\ref{Subexample: q = 5} \\
6 & 5 & $2 \cdot 2 \cdot 6 \cdot 42 \cdot 62 \cdot 2156$ &
$\checkmark$ \\
\rowcolor{lightgray}
6 & 5 & $2 \cdot 2 \cdot 8 \cdot 12 \cdot 386 \cdot 148224$ &
\Cref{Example: difference of squares}\ref{Subexample: q = 385 2} \\
6 & 5 & $2 \cdot 2 \cdot 8 \cdot 12 \cdot 404 \cdot 8166$ &
$\checkmark$ \\
\rowcolor{lightgray}
6 & 5 & $2 \cdot 2 \cdot 8 \cdot 12 \cdot 416 \cdot 5154$ &
$\checkmark$ \\
6 & 5 & $2 \cdot 2 \cdot 8 \cdot 12 \cdot 636 \cdot 974$ &
$\checkmark$ \\
\rowcolor{lightgray}
6 & 5 & $2 \cdot 2 \cdot 12 \cdot 14 \cdot 18 \cdot 206$ &
$\checkmark$ \\
6 & 9 & $2 \cdot 2 \cdot 2 \cdot 10 \cdot 82 \cdot 6560$ &
\Cref{Example: difference of squares}\ref{Subexample: q = 2^ell} \\
\rowcolor{lightgray}
6 & 9 & $2 \cdot 2 \cdot 2 \cdot 10 \cdot 92 \cdot 670$ &
$\checkmark$ \\
6 & 9 & $2 \cdot 2 \cdot 2 \cdot 10 \cdot 100 \cdot 422$ &
$\checkmark$ \\
\rowcolor{lightgray}
6 & 9 & $2 \cdot 2 \cdot 2 \cdot 10 \cdot 112 \cdot 290$ &
$\checkmark$ \\
6 & 11 & $2 \cdot 2 \cdot 2 \cdot 4 \cdot 34 \cdot 1088$ &
\Cref{Example: difference of squares}\ref{Subexample: q = 33} \\
\rowcolor{lightgray}
6 & 11 & $2 \cdot 2 \cdot 2 \cdot 4 \cdot 38 \cdot 244$ &
$\checkmark$ \\
6 & 11 & $2 \cdot 2 \cdot 2 \cdot 6 \cdot 8 \cdot 384$ &
\Cref{Example: difference of squares}\ref{Subexample: q = 385} \\
\rowcolor{lightgray}
6 & 13 & $2 \cdot 2 \cdot 2 \cdot 4 \cdot 8 \cdot 16$ &
$\checkmark$ \\
6 & 17 & $2 \cdot 2 \cdot 2 \cdot 2 \cdot 18 \cdot 288$ &
\Cref{Example: difference of squares}\ref{Subexample: q = 17} \\
\rowcolor{lightgray}
6 & 23 & $2 \cdot 2 \cdot 2 \cdot 2 \cdot 6 \cdot 6$ &
$\checkmark$ \\
6 & 33 & $2 \cdot 2 \cdot 2 \cdot 2 \cdot 2 \cdot 32$ &
\Cref{Example: Powers of 2 doubled}, \Cref{Example: difference of squares}\ref{Subexample: q = 2^ell} \\
\rowcolor{lightgray}
6 & 63 & $2 \cdot 2 \cdot 2 \cdot 2 \cdot 2 \cdot 2$ &
\Cref{Example: Powers of 2} \\
\end{tabular}
\end{table}

We have demonstrated that by relaxing the distinctness and primality conditions implicit in Lehmer's original conjecture, we uncover a rich structure of spoof Lehmer factorizations that showcase a variety of features. These examples provide natural barriers to certain proof strategies for Lehmer's conjecture: for example, any such proof must assume more about the factors $x_1 \dots x_r$ s of a putative Lehmer number $n$ than that they are positive, odd, and increasing, since for example $3 \cdot 5 \cdot 17 \cdot 285 \cdot 2507$ satisfies all of these assumptions.

\section{Future Work}\label{Section: Future work}

In this paper, we examined solutions to \Cref{Equation: Diophantine equation} with $x_1, \dots, x_r$ positive, continuing the fruitful program initiated in \cite{BYU-CNT}. It is natural to extend this search to allows for $x_1, \dots, x_r$ to be negative as well. The behavior of $L$ and $U$ become more delicate in this case, which complicates our analysis. Although we do not have a systematic analysis of \Cref{Equation: Diophantine equation} when negative bases are allowed, even an \text{ad hoc} investigation reveals some intriguing phenomena. We note first that $x \cdot (2 - x)$ is always a $1$-Lehmer factorization, so there are infinitely many odd Lehmer factorizations with 2 factors. In particular, $-3 \cdot 5$, $-5 \cdot 7$, and $-11 \cdot 13$ furnish easy examples of prime but non-positive solutions to \Cref{Equation: Diophantine equation}. Each pair $(p - 2, p)$ of twin primes yields the $1$-Lehmer factorization $p \cdot (2 - p)$, so the Twin Prime Conjecture would imply that we have infinitely many such prime examples. If we look instead for Lehmer factorizations with 3 factors and prime bases, we quickly stumble on the 1-Lehmer factorizations $-3 \cdot 5 \cdot 17$, $-5 \cdot 7 \cdot 37$, $-11 \cdot -7 \cdot 5$, $-11 \cdot 19 \cdot 31$, and more. For examples with 4 factors, consider the 1-Lehmer factorizations $-61 \cdot -5 \cdot 7 \cdot 23$ and $-17 \cdot 43 \cdot 47 \cdot 83$. We also have the 2-Lehmer factorizations $-7 \cdot 3 \cdot 3$ and $-223 \cdot 3 \cdot 5 \cdot 15$, but we know of no $k$-Lehmer factorizations consisting of distinct signed primes with $k \neq 1$. On the other hand, any true counterexample to Lehmer's totient conjecture must be a $k$-Lehmer factorization with $k > 1$.

The algorithm we set forth in this paper can certainly be optimized further, but even in its present na\"ive form we expect more interesting data could be generated by an entrepreneurial spirit with more computational power or patience than we possess.

Finally, our paper joins \cite{BYU-CNT} in taking open problems regarding arithmetic functions and exploring how these problems change when we expand their scope from distinct positive prime arguments to more general integral arguments. The literature is littered with problems of this sort \cite{Carmichael, hagis_cohen_1982, Kishore, Ore}, and we expect them to provide a fertile field for forthcoming mathematicians who are inclined to take such a roundabout route into arithmetic geometry.


\begin{thebibliography}{10}

\bibitem{BYU-CNT}
Nickolas Andersen, Spencer Durham, Michael~J. Griffin, Jonathan Hales, Paul
  Jenkins, Ryan Keck, Hankun Ko, Grant Molnar, Eric Moss, Pace~P. Nielsen, Kyle
  Niendorf, Vandy Tombs, Merrill Warnick, and Dongsheng Wu, \emph{Odd, spoof
  perfect factorizations}, Journal of Number Theory (2021).

\bibitem{burcsi2011computational}
P{\'e}ter Burcsi, S{\'a}ndor Czirbusz, and G{\'a}bor Farkas,
  \emph{Computational investigation of {L}ehmer’s totient problem}, Ann.
  Univ. Sci. Budapest. Sect. Comput \textbf{35} (2011), 43--49.

\bibitem{Carmichael}
Robert~D. Carmichael, \emph{On {E}uler's $\phi$-function}, Bull. Amer. Math.
  Soc. \textbf{13} (1907), 241--243.

\bibitem{Cohen-Hagis}
Graeme~L. Cohen and Peter Hagis, \emph{On the number of prime factors of n if
  $\varphi(n) \mid (n - 1)$}, Nieuw Arch. Wiskd., III. Ser. 28 \textbf{28}
  (1980), 177--185.

\bibitem{hagis_cohen_1982}
\bysame, \emph{Some results concerning quasiperfect numbers}, Journal of the
  Australian Mathematical Society. Series A. Pure Mathematics and Statistics
  \textbf{33} (1982), no.~2, 275–286.

\bibitem{Dittmer}
Samuel Dittmer, \emph{Spoof odd perfect numbers}, Math. Comp. \textbf{83}
  (2041), 2575--2582.

\bibitem{Hagis}
Peter Hagis, \emph{On the equation $n \cdot \phi(n) = n - 1$}, Nieuw Arch.
  Wiskd., IV. Ser. 6 \textbf{3} (1988), 255--261.

\bibitem{Kishore}
Masao Kishore, \emph{On odd perfect, quasiperfect, and odd almost perfect
  numbers}, Math. Comp. \textbf{36} (1981), 583--586.

\bibitem{Lehmer1932}
Derrick~Henry Lehmer, \emph{On the {E}uler totient function}, Bulletin of the
  AMS (1932), 0--8.

\bibitem{Luca-Pomerance}
Florian Luca and Carl Pomerance, \emph{On composite integers $n$ for which
  $\phi(n) \mid n - 1$}, Bol. Soc. Mat. Mexicana \textbf{3} (2011), 13--21.

\bibitem{Ore}
Oystein Ore, \emph{On the averages of the divisors of a number}, The American
  Mathematical Monthly \textbf{55} (1948), 615--619.

\end{thebibliography}

\providecommand{\bysame}{\leavevmode\hbox to3em{\hrulefill}\thinspace}
\providecommand{\MR}{\relax\ifhmode\unskip\space\fi MR }
\providecommand{\MRhref}[2]{%
  \href{http://www.ams.org/mathscinet-getitem?mr=#1}{#2}
}
\providecommand{\href}[2]{#2}

\end{document}